\documentclass[11pt]{article}
\usepackage{amsmath, amssymb, amsthm}
\usepackage[utf8]{inputenc}
\usepackage[T1]{fontenc}
\usepackage{geometry}
\usepackage{hyperref}

\newtheorem{theorem}{Theorem}
\newtheorem{proposition}[theorem]{Proposition}

\theoremstyle{remark}
\newtheorem{remark}[theorem]{Remark}

\newcommand{\zb}{\bar z}

\title{The complex form of Vekua's characteristic factor: a derivation, and two sign corrections in §7 of Generalized Analytic Functions}
\author{Daniel Alay\'on-Solarz\thanks{danieldaniel@gmail.com}}
\date{June 2026}

\begin{document}

\maketitle

\begin{abstract}
In \S7 of \emph{Generalized Analytic Functions} \cite{vekua}, the reduction of a
first-order elliptic system to canonical form proceeds through a factor of the
characteristic equation, which Vekua selects in real form~(7.13) and then
restates, without derivation, in complex (Beltrami) form~(7.14). We supply that
conversion. With the standard Wirtinger convention used below, the complex form
of~(7.13) is the negative of the coefficient printed in~(7.14) (p.~126, 1962
Pergamon edition), and we confirm the correct sign against Vekua's own
factorization~(7.12) and his canonical coefficient~(7.17). A related sign defect
appears in the second-order Beltrami coefficient~(7.23) (p.~127): a coordinate
solving (7.23) as printed reduces the equation to the canonical form~(7.26) only
in the special symmetric case $a=c$. In both instances the error is confined to
the displayed coefficient and leaves the surrounding reduction, carried out
independently of it, intact; we record the corrected coefficient in each case.
\end{abstract}

\section{The reduction}\label{sec:setup}

We recall the relevant equations of \cite[\S7]{vekua} in the author's notation;
all equation numbers in parentheses are Vekua's. The starting point is a
first-order elliptic system in real unknowns $u,v$, reduced at (7.5) to
\[
\begin{cases}
-v_y + a_{11}u_x + a_{12}u_y + a_1 u + b_1 v = f_1,\\
\phantom{-}v_x + a_{21}u_x + a_{22}u_y + a_2 u + b_2 v = f_2,
\end{cases}
\]
elliptic in the sense (7.6) that, with
\begin{equation}\label{eq:Delta}
\Delta := a_{11}a_{22} - \tfrac14(a_{12}+a_{21})^2,
\end{equation}
one has $a_{11}>0$ and $\Delta>0$.

To bring the system to canonical form one introduces a change of variables
$\zeta = \xi + i\eta$, $\zeta = \zeta(x,y)$, with Jacobian
$J := \xi_x\eta_y - \xi_y\eta_x \neq 0$ (Vekua's (7.10)). The principal
coefficients transform, by Vekua's formulas on p.~125, to
\begin{align}
a'_{11} &= \tfrac1J\bigl(a_{11}\xi_x^2 + (a_{12}+a_{21})\xi_x\xi_y + a_{22}\xi_y^2\bigr),\notag\\
a'_{12} &= \tfrac1J\bigl(a_{11}\xi_x\eta_x + a_{12}\xi_x\eta_y + a_{21}\xi_y\eta_x + a_{22}\xi_y\eta_y\bigr),\label{eq:aprime}\\
a'_{21} &= \tfrac1J\bigl(a_{11}\xi_x\eta_x + a_{12}\xi_y\eta_x + a_{21}\xi_x\eta_y + a_{22}\xi_y\eta_y\bigr),\notag\\
a'_{22} &= \tfrac1J\bigl(a_{11}\eta_x^2 + (a_{12}+a_{21})\eta_x\eta_y + a_{22}\eta_y^2\bigr).\notag
\end{align}
The canonical form is reached when (7.11) holds,
\begin{equation}\label{eq:canon}
a'_{11} = a'_{22}, \qquad a'_{12} = -a'_{21},
\end{equation}
which, written for $\zeta = \xi + i\eta$, is the single complex equation (7.12),
\begin{equation}\label{eq:char}
a_{11}\zeta_x^2 + (a_{12}+a_{21})\zeta_x\zeta_y + a_{22}\zeta_y^2 = 0 .
\end{equation}
Vekua selects the factor (7.13),
\begin{equation}\label{eq:factor}
a_{11}\zeta_x + \Bigl(\tfrac12(a_{12}+a_{21}) + i\sqrt{\Delta}\Bigr)\zeta_y = 0,
\end{equation}
and states it ``in the complex notation'' as the Beltrami equation (7.14),
\begin{equation}\label{eq:714}
\zeta_{\zb} - q\,\zeta_z = 0,
\qquad
q = \frac{a_{11} - \sqrt{\Delta} + \tfrac{i}{2}(a_{12}+a_{21})}
{a_{11} + \sqrt{\Delta} - \tfrac{i}{2}(a_{12}+a_{21})}\,.
\end{equation}
From (7.13) he then derives (7.15)--(7.16) and, in particular, (7.17),
\begin{equation}\label{eq:717}
a'_{11} = \frac{i\sqrt{\Delta}}{2J}\,(\zeta_x\bar\zeta_y - \bar\zeta_x\zeta_y) = \sqrt{\Delta}\,.
\end{equation}

\begin{remark}
The misprint reported here is observed in the 1962 Pergamon English
translation~\cite{vekua}. We have not collated \eqref{eq:714} against Vekua's
Russian original, and we therefore do not determine whether the sign error
originates with the author or was introduced in translation or typesetting.
The argument of this note settles the error in the English edition
independently of that question, since it turns only on the internal
consistency of \eqref{eq:factor}, \eqref{eq:char}, and (7.17).
\end{remark}

\section{The discrepancy}\label{sec:discrepancy}

We compute the complex form of (7.13) directly. Throughout, $\zeta_z$ and
$\zeta_{\zb}$ are the Wirtinger derivatives
$\zeta_z = \tfrac12(\zeta_x - i\zeta_y)$,
$\zeta_{\zb} = \tfrac12(\zeta_x + i\zeta_y)$, equivalently
\begin{equation}\label{eq:wirtinger}
\zeta_x = \zeta_z + \zeta_{\zb}, \qquad \zeta_y = i\,(\zeta_z - \zeta_{\zb}),
\end{equation}
the convention in which the quotient $q$ of (7.14) is itself written.

\begin{proposition}\label{prop:convert}
In complex form the factor \eqref{eq:factor} reads $\zeta_{\zb} - q_\star\,\zeta_z = 0$
with
\[
q_\star = -\,\frac{a_{11} - \sqrt{\Delta} + \tfrac{i}{2}(a_{12}+a_{21})}
{a_{11} + \sqrt{\Delta} - \tfrac{i}{2}(a_{12}+a_{21})}\,.
\]
That is, $q_\star = -q$, the negative of the coefficient printed in
\eqref{eq:714}.
\end{proposition}

\begin{proof}
Write $P := a_{12}+a_{21}$. Substituting \eqref{eq:wirtinger} into
\eqref{eq:factor},
\[
a_{11}(\zeta_z + \zeta_{\zb}) + \Bigl(\tfrac{P}{2} + i\sqrt{\Delta}\Bigr)\,
i\,(\zeta_z - \zeta_{\zb}) = 0 .
\]
Collecting $\zeta_z$ and $\zeta_{\zb}$ and using $i\cdot i\sqrt\Delta = -\sqrt\Delta$,
\[
\Bigl[\,(a_{11} - \sqrt{\Delta}) + \tfrac{i}{2}P\,\Bigr]\zeta_z
\;+\;
\Bigl[\,(a_{11} + \sqrt{\Delta}) - \tfrac{i}{2}P\,\Bigr]\zeta_{\zb}
\;=\; 0 .
\]
The bracket on $\zeta_{\zb}$ is nonzero (its real part $a_{11}+\sqrt\Delta>0$), so
\[
\zeta_{\zb} = -\,\frac{(a_{11} - \sqrt{\Delta}) + \tfrac{i}{2}P}
{(a_{11} + \sqrt{\Delta}) - \tfrac{i}{2}P}\,\zeta_z ,
\]
which is the stated $q_\star$.
\end{proof}

Both $q$ and $q_\star = -q$ have modulus
\[
|q| = |q_\star| = \frac{\sqrt{(a_{11}-\sqrt{\Delta})^2 + \tfrac14 P^2}}
{\sqrt{(a_{11}+\sqrt{\Delta})^2 + \tfrac14 P^2}} < 1 ,
\]
so the discrepancy is not one of modulus or of ellipticity: it is purely a sign.
With the Wirtinger convention \eqref{eq:wirtinger}, the complex form of
\eqref{eq:factor} is exactly the coefficient $q_\star$ above. For completeness,
if one reverses the sign convention for the $y$-derivative, writing instead
$\zeta_y=-i(\zeta_z-\zeta_{\zb})$, the corresponding quotient is
\[
\zeta_{\zb}
=
-\frac{(a_{11}+\sqrt\Delta)-\tfrac i2 P}
{(a_{11}-\sqrt\Delta)+\tfrac i2 P}\,\zeta_z ,
\]
which has modulus exceeding $1$. Thus the printed coefficient in \eqref{eq:714}
is not recovered by a harmless convention change.

\section{Which sign is correct}\label{sec:resolution}

Two of Vekua's own equations fix the sign as $q_\star = -q$.

\medskip\noindent\textbf{(a) The factorization (7.12).}
Equation \eqref{eq:factor} is genuinely a factor of the characteristic
\eqref{eq:char}: with $P = a_{12}+a_{21}$ and $c_\pm := (\tfrac{P}{2} \pm i\sqrt\Delta)/a_{11}$,
\[
a_{11}(\zeta_x + c_+\zeta_y)(\zeta_x + c_-\zeta_y)
= a_{11}\zeta_x^2 + a_{11}(c_++c_-)\zeta_x\zeta_y + a_{11}c_+c_-\,\zeta_y^2 ,
\]
and $c_++c_- = P/a_{11}$, while
$c_+c_- = \bigl(\tfrac{P^2}{4}+\Delta\bigr)/a_{11}^2 = a_{22}/a_{11}$ by
\eqref{eq:Delta}; thus the product is exactly \eqref{eq:char}. So
\eqref{eq:factor} is the equation Vekua intends, and a coordinate solving it
satisfies the canonical conditions \eqref{eq:canon}.

\medskip\noindent\textbf{(b) The canonical coefficient (7.17).}
Vekua's own passage from (7.13) to (7.17) fixes the branch in print: the two
relations he displays on p.~126, just before (7.17),
\[
a_{11}\zeta_x + \tfrac12(a_{12}+a_{21})\,\zeta_y = -\,i\sqrt{\Delta}\,\zeta_y,
\qquad
a_{22}\zeta_y + \tfrac12(a_{12}+a_{21})\,\zeta_x = i\sqrt{\Delta}\,\zeta_x,
\]
are the real factor \eqref{eq:factor} and its $a_{22}$-companion (equivalent to it
modulo \eqref{eq:char}), both carrying $+i\sqrt{\Delta}$, and it is by introducing
them into (7.16) that he reaches (7.17)---so the branch underlying his own
canonical coefficient is exactly the one whose honest complex form is
$q_\star=-q$, not the printed $q$.
A coordinate solving \eqref{eq:factor} returns Vekua's value $a'_{11}=+\sqrt\Delta$.
Indeed, the real and imaginary parts of \eqref{eq:factor} are
\[
a_{11}\xi_x + \tfrac{P}{2}\xi_y - \sqrt{\Delta}\,\eta_y = 0,
\qquad
a_{11}\eta_x + \tfrac{P}{2}\eta_y + \sqrt{\Delta}\,\xi_y = 0,
\]
so that
\[
\xi_x = \frac{-\tfrac{P}{2}\xi_y + \sqrt{\Delta}\,\eta_y}{a_{11}},
\qquad
\eta_x = \frac{-\tfrac{P}{2}\eta_y - \sqrt{\Delta}\,\xi_y}{a_{11}} .
\]
Substituting into \eqref{eq:aprime} and using $4a_{11}a_{22}-P^2 = 4\Delta$,
\[
J = \frac{(\xi_y^2+\eta_y^2)\sqrt{\Delta}}{a_{11}} > 0,
\qquad
a'_{11} = a'_{22} = \sqrt{\Delta},
\qquad
a'_{12} = -a'_{21} = \tfrac12(a_{12}-a_{21}),
\]
identically in $(\xi_y,\eta_y)$. This reproduces (7.17) and the canonical
conditions \eqref{eq:canon}, with the leading coefficient \emph{positive}.

\medskip
The printed coefficient $q$ of \eqref{eq:714} is excluded by~(a) alone. Its
ratio $\zeta_x/\zeta_y = -i(1+q)/(1-q)$ is not a root of the characteristic
\eqref{eq:char} (it is a root only in the conformal case $\Delta = a_{11}^2$,
$P=0$, where $q = q_\star = 0$ and the discrepancy disappears). Hence, although
$|q|<1$ so that $\zeta_{\zb} = q\,\zeta_z$ is a bona fide Beltrami equation, it
solves \emph{neither} factor of \eqref{eq:char}: a coordinate satisfying it does
not meet the canonical conditions \eqref{eq:canon} and yields no canonical form
at all. The printed $q$ is therefore not the conjugate factor $-i\sqrt{\Delta}$
of \eqref{eq:char}---that factor has modulus exceeding $1$---and it does not
correspond to either characteristic factor.

Computation~(b) plays a complementary role: it shows that the branch
$\pm i\sqrt{\Delta}$ selected in the real factor fixes the sign of the leading
coefficient. The branch $+i\sqrt{\Delta}$ of \eqref{eq:factor} returns
$a'_{11} = +\sqrt{\Delta}$, whereas the conjugate branch $-i\sqrt{\Delta}$ would
return $a'_{11} = -\sqrt{\Delta}$, in conflict with~(7.17). This singles out
$+i\sqrt{\Delta}$, and hence \eqref{eq:factor} itself, as the factor Vekua
intends. Both of his own equations therefore require its honest complex form,
$q_\star = -q$.

\begin{remark}
The misprint is confined to the display \eqref{eq:714}. The subsequent computation
(7.15)--(7.17), and the canonical reduction it leads to, are carried out from the
real factor \eqref{eq:factor}; they are correct as printed and are unaffected by
the sign in \eqref{eq:714}.
\end{remark}

\section{The companion misprint in (7.23)}\label{sec:723}

A related sign defect appears one page later, in the reduction of a second-order
equation. For the elliptic equation (7.21),
\begin{equation}\label{eq:721}
a\,u_{xx} + 2b\,u_{xy} + c\,u_{yy}
   + F\!\left(x,y,u,u_x,u_y\right) = 0,
\qquad \Delta := ac - b^2 > 0,\quad a>0,
\end{equation}
Vekua reduces to canonical form by a homeomorphism $\zeta = \xi + i\eta = \zeta(z)$
of the Beltrami equation (7.23), printed (p.~127) as
\begin{equation}\label{eq:723}
\zeta_{\zb} - q\,\zeta_z = 0,
\qquad
q = \frac{a - \sqrt{\Delta} - ib}{a + \sqrt{\Delta} + ib}\,.
\end{equation}
On p.~128 he records the consequence he requires of \eqref{eq:723}: for the
linear equation (7.24), the change of variables (7.25) with $\xi+i\eta$ a
homeomorphism of \eqref{eq:723} produces (7.26),
\[
u_{\xi\xi} + u_{\eta\eta} + (\text{lower-order terms}) = 0,
\]
whose principal part is the Laplacian. As with \eqref{eq:714}, equation
\eqref{eq:723} is asserted, not derived, on these pages; we test it against the
reduction it is meant to effect.

Under any change of variables $\xi=\xi(x,y),\ \eta=\eta(x,y)$ the principal part
of \eqref{eq:721} transforms to
$A\,u_{\xi\xi}+2B\,u_{\xi\eta}+C\,u_{\eta\eta}$ with
\begin{equation}\label{eq:ABC}
\begin{aligned}
A &= a\,\xi_x^2 + 2b\,\xi_x\xi_y + c\,\xi_y^2,\\
B &= a\,\xi_x\eta_x + b\,(\xi_x\eta_y+\xi_y\eta_x) + c\,\xi_y\eta_y,\\
C &= a\,\eta_x^2 + 2b\,\eta_x\eta_y + c\,\eta_y^2,
\end{aligned}
\end{equation}
and (7.26) demands $A = C$ and $B = 0$; ellipticity gives $A=C>0$, and division
by $A$ then normalizes the leading part to $u_{\xi\xi}+u_{\eta\eta}$ and produces
the lower-order coefficients (7.27). Writing $\zeta=\xi+i\eta$, the two real
conditions combine into the single complex equation
\begin{equation}\label{eq:char2}
(A-C) + 2iB \;=\; a\,\zeta_x^2 + 2b\,\zeta_x\zeta_y + c\,\zeta_y^2 \;=\; 0,
\end{equation}
which is exactly the characteristic \eqref{eq:char} under the identification
\begin{equation}\label{eq:dict}
a_{11}\leftrightarrow a,\qquad
\tfrac12(a_{12}+a_{21})\leftrightarrow b,\qquad
a_{22}\leftrightarrow c,
\end{equation}
beneath which $\Delta = a_{11}a_{22}-\tfrac14(a_{12}+a_{21})^2$ becomes $ac-b^2$.

\begin{proposition}\label{prop:723}
Let $a>0$ and $\Delta = ac-b^2>0$. At points where $d\zeta\neq 0$, a coordinate
$\zeta=\xi+i\eta$ reduces the principal part of \eqref{eq:721} to a scalar
multiple of $u_{\xi\xi}+u_{\eta\eta}$ if and only if its differential satisfies
\begin{equation}\label{eq:char2pointwise}
a\,\zeta_x^2 + 2b\,\zeta_x\zeta_y + c\,\zeta_y^2 = 0 .
\end{equation}
The unique Beltrami coefficient with $|q|<1$ for which
$\zeta_{\zb}=q\,\zeta_z$ has this property is
\begin{equation}\label{eq:723corrected}
q \;=\; -\,\frac{a - \sqrt{\Delta} + ib}{a + \sqrt{\Delta} - ib}
   \;=\; \frac{c - a - 2ib}{a + c + 2\sqrt{\Delta}}\,.
\end{equation}
The coefficient printed in \eqref{eq:723} equals \eqref{eq:723corrected} only
when $a=c$; for $a\neq c$ it gives $A\neq C$, so a coordinate solving
\eqref{eq:723} does not reduce {\rm(7.24)} to the canonical form {\rm(7.26)}.
\end{proposition}

\begin{proof}
The expressions \eqref{eq:ABC} are the standard transformation of a second-order
principal part, and \eqref{eq:char2} is their complex combination, identical to
\eqref{eq:char} under \eqref{eq:dict}. The factor of \eqref{eq:char2} carrying
$+i\sqrt\Delta$, namely $a\,\zeta_x + (b+i\sqrt\Delta)\,\zeta_y = 0$, is
\eqref{eq:factor} under \eqref{eq:dict}; Proposition~\ref{prop:convert} converts
it to $\zeta_{\zb} = q_\star\zeta_z$ with $q_\star$ equal to the right-hand side
of \eqref{eq:723corrected}, and $|q_\star|<1$. The conjugate factor gives the
reciprocal coefficient, of modulus exceeding $1$; hence \eqref{eq:723corrected}
is the unique admissible coefficient.

For the printed \eqref{eq:723}, it is enough to test the induced differential
pointwise. The conditions $A=C$, $B=0$ are pointwise conditions on
$d\zeta$ and are invariant under multiplication of $d\zeta$ by a nonzero complex
scalar. We may therefore freeze the coefficients at a point and use the affine
representative $\zeta=z+q\,\bar z$, for which
\[
\xi_x=1+\operatorname{Re}q,\qquad
\xi_y=\eta_x=\operatorname{Im}q,\qquad
\eta_y=1-\operatorname{Re}q .
\]
A direct computation with $q$ as printed in \eqref{eq:723} yields
\[
A - C \;=\;
\frac{4\,(a-c)\,\bigl[(a-c)^2 + 4\Delta + 2(a+c)\sqrt{\Delta}\,\bigr]}
{\bigl(a+c+2\sqrt{\Delta}\,\bigr)^2}\,.
\]
The bracketed factor and the denominator are positive, so $A-C$ has the sign of
$a-c$; it vanishes only in the special case $a=c$. The coefficient
\eqref{eq:723corrected}, by contrast, gives $A-C=0$ and $B=0$ identically.
\end{proof}

\begin{remark}
The two misprints are related but not identical. Printed \eqref{eq:714} is the
negative of the correct coefficient: a clean sign error. Printed \eqref{eq:723}
is $-\overline{q_\star}$: it carries the wrong sign on its real part alone, and
agrees with \eqref{eq:723corrected} only in the special symmetric case $a=c$.
The two printed coefficients are complex conjugates of one another under
\eqref{eq:dict}. This conjugation reflects the bookkeeping between the two
standard ways of writing the characteristic: the slope form
$a\,dy^2 - 2b\,dx\,dy + c\,dx^2 = 0$ used for second-order equations carries
$-2b$, against the term
$+(a_{12}+a_{21})\,\zeta_x\zeta_y$ in \eqref{eq:char}. The intervening
$b\mapsto -b$ accounts for the conjugation. In both displays, the printed
Beltrami coefficient is not the one obtained from the real characteristic factor
used in the surrounding reduction.
\end{remark}

\section{Corrected statements}\label{sec:corrected}

\subsection{Equation (7.14)}
Equation (7.14) of \cite{vekua} should read
\[
\zeta_{\zb} - q\,\zeta_z = 0,
\qquad
q = -\,\frac{a_{11} - \sqrt{\Delta} + \tfrac{i}{2}(a_{12}+a_{21})}
{a_{11} + \sqrt{\Delta} - \tfrac{i}{2}(a_{12}+a_{21})}\,,
\]
equivalently
\[
q = \frac{\sqrt{\Delta} - a_{11} - \tfrac{i}{2}(a_{12}+a_{21})}
{a_{11} + \sqrt{\Delta} - \tfrac{i}{2}(a_{12}+a_{21})}\,,
\]
the complex form of the factor (7.13).

\subsection{Equation (7.23)}
Likewise, equation (7.23) of \cite{vekua} should read
\[
\zeta_{\zb} - q\,\zeta_z = 0,
\qquad
q = -\,\frac{a-\sqrt{\Delta}+ib}{a+\sqrt{\Delta}-ib}
  = \frac{c-a-2ib}{a+c+2\sqrt{\Delta}},
\qquad \Delta = ac-b^2,
\]
the complex form of the characteristic factor
$a\,\zeta_x + \bigl(b+i\sqrt{\Delta}\bigr)\zeta_y = 0$.

\subsection*{Use of generative AI tools}

The author discloses
the use of Anthropic's Claude (Claude Opus 4.8, accessed through the Claude.ai web
interface, June 2026) in the preparation of this note. The inconsistency between
equations (7.13) and (7.14) was identified in the course of dialogue with the
tool; in a subsequent review of the manuscript the tool further identified that
the second-order coefficient (7.23) carries a related defect, and that the
consistency of (7.14) with (7.23) was a question the note had to resolve, which
prompted Section~\ref{sec:723}.

The conversion of Proposition~\ref{prop:convert}, the factorization and the
coefficient computation of Section~\ref{sec:resolution}, and the canonical-form
criterion and failure computation of Proposition~\ref{prop:723}---the conditions
$A=C$, $B=0$ and the explicit expression for $A-C$---were stress-tested by exact
rational computer algebra (SymPy), with numerical cross-checks (NumPy) on
representative coefficients; the scripts were produced with the tool's assistance
and re-run independently by the author. The prose was drafted in iterative
dialogue, with all statements, computations, and their wording reviewed by the
author. The author takes full responsibility for the correctness, accuracy,
originality, and integrity of all content.

\subsection*{Disclosure of interest}

The author reports there are no competing interests to declare.

\end{document}